\documentclass[11pt]{amsart}
\usepackage[usenames]{color}
\usepackage{fullpage}
\usepackage{amscd}
\usepackage{amssymb, latexsym}
\usepackage[all,2cell,ps]{xy}
\usepackage{mathdots}

\theoremstyle{plain}
\newtheorem{thm}{Theorem}[section]
\newtheorem{theorem}[thm]{Theorem}

\newtheorem{corollary}[thm]{Corollary}

\newtheorem{lemma}[thm]{Lemma}

\newtheorem{proposition}[thm]{Proposition}

\theoremstyle{definition}
\newtheorem{de}[thm]{Definition}

\newtheorem{remark}[thm]{Remark}
\newtheorem{example}[thm]{Example}

\newcommand{\Z}{\mathbb{Z}}

\newcommand{\id}{\mathrm{id}}

\newcommand{\ld}{\backslash}

\newcommand{\Mlt}{\mathop{\mathrm{Mlt}}}
\newcommand{\Aut}{\mathop{\mathrm{Aut}}}

\newcommand{\LMlt}{\mathop{\mathrm{LMlt}}}
\newcommand{\RMlt}{\mathop{\mathrm{RMlt}}}
\newcommand{\eqrel}[2]{\mathrel{\stackrel{\scriptstyle\eqref{#1}}{#2}}}
\numberwithin{equation}{section}

\begin{document}

\title{Distributive biracks and solutions of the Yang-Baxter equation}

\author{P\v remysl Jedli\v cka}
\author{Agata Pilitowska}
\author{Anna Zamojska-Dzienio}

\address{(P.J.) Department of Mathematics, Faculty of Engineering, Czech University of Life Sciences, Kam\'yck\'a 129, 16521 Praha 6, Czech Republic}
\address{(A.P., A.Z.) Faculty of Mathematics and Information Science, Warsaw University of Technology, Koszykowa 75, 00-662 Warsaw, Poland}

\email{(P.J.) jedlickap@tf.czu.cz}
\email{(A.P.) A.Pilitowska@mini.pw.edu.pl}
\email{(A.Z.) A.Zamojska-Dzienio@mini.pw.edu.pl}


\date{\today}

\begin{abstract}
We investigate a class of non-involutive solutions of the Yang-Baxter equation which generalize derived (self-distributive) solutions. In particular, we study generalized multipermutation solutions in this class. We show that the Yang-Baxter (permutation) groups of such solutions are nilpotent. We formulate the results in the language of biracks which allows us to apply universal algebra tools.
\end{abstract}

\maketitle
\section{Introduction}
The Yang-Baxter equation is a fundamental equation occurring in integrable models in statistical mechanics and quantum field theory~\cite{Jimbo}.
Let $V$ be a vector space. A {\em solution of the Yang--Baxter equation} is a linear mapping $r:V\otimes V\to V\otimes V$ such that
\[
(id\otimes r) (r\otimes id) (id\otimes r)=(r\otimes id) (id\otimes r) (r\otimes id).
\]
Since description of all possible solutions seems to be extremely difficult, Drinfeld \cite{Dr90} introduced the following simplification.

Let $X$ be a basis of the space $V$ and let $\sigma:X^2\to X$ and $\tau: X^2\to X$ be two mappings.
We say that $(X,\sigma,\tau)$ is a {\em set-theoretic solution of the Yang--Baxter equation} if
the mapping $x\otimes y \mapsto \sigma(x,y)\otimes \tau(x,y)$ extends to a solution of the Yang--Baxter
equation. It means that $r\colon X^2\to X^2$, where $r=(\sigma,\tau)$ satisfies the \emph{braid relation}:

\begin{equation}\label{eq:braid}
(id\times r)(r\times id)(id\times r)=(r\times id)(id\times r)(r\times id).
\end{equation}

A solution is called {\em non-degenerate} if the mappings $\sigma(x,\_)=\sigma_x:X\to X$ and $\tau(\_\,,y)=\tau_y:X\to X$ are bijections,
for all $x,y\in X$.
A solution $(X,\sigma,\tau)$ is {\em involutive} if $r^2=\mathrm{id}_{X^2}$, and it is
\emph{square free} if $r(x,x)=(x,x)$, for every $x\in X$.

In \cite[Section 3.2]{ESS} Etingof, Schedler and Soloviev introduced, for each involutive solution $(X,\sigma,\tau)$, the equivalence relation $\sim$ on the set $X$: for each $x,y\in X$

\[
x\sim y\quad \Leftrightarrow\quad \sigma_x=\sigma_y.
\]
They showed that the quotient set $X/\mathord{\sim}$ can be again endowed
with a structure of an involutive solution. This does not work for non-involutive solutions. In \cite{JPZ} we showed that in the non-involutive case the role similar to the relation $\sim$ is played by the relation $\approx$ defined on the set $X$ as follows:
for each $x,y\in X$

\[
x\approx y\quad \Leftrightarrow\quad \sigma_x=\sigma_y\quad\text{and}\quad\tau_x=\tau_y.
\]
We call a solution obtained on the set $X/\mathord{\approx}$, the {\em retraction} of the solution~$X$ and denote it by
$\mathrm{Ret}(X)$. A solution~$X$ is said to be a multipermutation solution of level~$k$, if $k$ is the smallest integer
such that $|\mathrm{Ret}^k(X)|=1$.

Two types of solutions are particularly well studied: involutive solutions and \emph{derived} solutions, i.e. those with all 
$\sigma_x$ or all $\tau_y$ being the identity mappings. In \cite{LV} Lebed and Vendramin have thoroughly investigated \emph{injective} solutions, which generalize involutive ones. In this paper, we focus on generalization of derived solutions given by \emph{distributive} solutions.

The Yang-Baxter group of a solution $(X,\sigma,\tau)$ is the group generated by all bijections $\sigma_x$ and $\tau_x$, for $x\in X$. 
There were several results for involutive solutions
connecting properties of the Yang-Baxter group and
multipermutation level of the solution \cite{CJO10,CJO14,GI18,GIC12,MBE,Rump07,Smo19}.

Let now $(X,\sigma,\tau)$ be a derived solution such that all $\sigma_x=\id_X$. Then $\tau_y\tau_x=\tau_{\tau_y(x)}\tau_y$ for all $x,y\in X$. Moreover, this condition holds for each element $\eta$ of the Yang-Baxter group, i.e. $\eta\tau_x=\tau_{\eta(x)}\eta$ for all $x\in X$. For derived solutions with all $\tau_x=\id_X$, one obtains the dual situation. Here, we consider solutions which are not-necessarily derived, but for each element $\eta$ in their Yang-Baxter group one has:
for each $x\in X$
\begin{equation}\label{e:aut}
\eta\sigma_x=\sigma_{\eta(x)}\eta\quad\text{and}\quad\eta\tau_x=\tau_{\eta(x)}\eta.
\end{equation}
These are called \emph{distributive} solutions. Condition \eqref{e:aut} means that each $\eta$ is actually an \emph{automorphism} of a solution (see \cite[Section 2]{GIM07}).

In \cite{JPZ19} we described the involutive distributive solutions. They are always multipermutation solutions of level 2 and their (involutive) Yang-Baxter groups are always abelian \cite[Theorem 7.6]{JPZ19}. In this paper we focus on non-involutive case. The situation is now more complex.

\medskip

\textbf{Main Theorem.}

\emph{Let $(X,\sigma,\tau)$ be a non-degenerate distributive solution of Yang-Baxter equation and let $k>1$. Then $X$ is a multipermutation solution of level at most $k$ if and only if the Yang-Baxter group of $X$ is nilpotent of class at most $k-1$.}

\medskip

This theorem cannot be generalized for non-distributive
solutions as there exist on one hand involutive solutions that are multipermutation but their Yang-Baxter groups
are not nilpotent \cite[Remark 7]{Smo19} and on the other hand there exist involutive solutions that are not multipermutation
but their Yang-Baxter groups are nilpotent \cite[Remark 6]{Smo19}.

It is known (see e.g. \cite{FJSK, S06, D15}) that there is a one-to-one correspondence between solutions of the Yang-Baxter equation and \emph{biracks} $(X,\circ,\ld_{\circ},\bullet,/_{\bullet})$ -- structures which satisfy some additional conditions \eqref{eq:lq}--\eqref{eq:b3}. This fact allows us to prove the Main Theorem using the language of biracks (Theorem \ref{thm:main}).

In particular, the correspondence exists between derived solutions and \emph{racks} \cite{ESG}. For each right rack $(X,\ast,\ld_\ast)$ its \emph{right multiplication group} $\RMlt(X)=\langle R_x\colon X\to X; a\mapsto a\ast x\mid\ x\in X\rangle$ is a subgroup of the automorphism group $\Aut(X)$. The similar property occurs for a left rack, and the reason for that in both cases, is a one-sided (self)distributivity. Therefore, we investigate biracks with at least one (self)distributive operation.  

The paper is organized as follows:
in Section 2 we characterize distributive biracks. We also give examples (Examples \ref{ex:wada} and \ref{ex:6}) of non-involutive distributive biracks which are not derived ones.
Section 3 is devoted to the quotient of distributive biracks by the relation $\approx$ 
as well as by the relation~$\sim$, that turns out to be
congruence as well, in the distributive case (Theorem \ref{th:left_congr}).
We also show (Lemma \ref{lm:sf}) that
the quotient birack by the congruence~$\sim$ is always idempotent and derived.
The last Section 4 contains the main result of the paper (Theorem \ref{thm:main}). We prove there that a distributive birack is of multipermutation level~$m$, for $m\geq 2$, if and only if its permutation group is nilpotent of class at most $m-1$.

\section{Distributive biracks}\label{sec:DB}
As we already mentioned there is a one-to-one correspondence between solutions of the Yang-Baxter equation and algebraic structures called biracks, which naturally appear in low-dimensional topology \cite{FJSK, EN}. In \cite{Kau99} Kauffman introduced the \emph{virtual knot theory} --- biracks play there a similar role as racks in classical knot theory. However, still not much is known about their structure and properties.

The following equational definition of a birack was given first by Stanovsk\'y in \cite{S06} (see also \cite{FJSK}).
\begin{de}
A structure $(X,\circ,\ld_{\circ},\bullet,/_{\bullet})$ with four binary operations is called a {\em birack}, if the following holds for any $x,y,z\in X$:
 \begin{align}
x\circ(x\ld_{\circ} y)=y=x\ld_{\circ} (x\circ y),\label{eq:lq}\\
(y/_{\bullet} x)\bullet x=y=(y\bullet x)/_{\bullet} x,\label{eq:rq}\\
x\circ(y\circ z)=(x\circ y)\circ((x\bullet y)\circ z),\label{eq:b1}\\
(x\circ y)\bullet((x\bullet y)\circ z)=(x\bullet(y\circ z))\circ(y\bullet z), \label{eq:b2}\\
(x\bullet y)\bullet z= (x\bullet (y\circ z))\bullet (y \bullet z).\label{eq:b3}
  \end{align}
\end{de}

\begin{example}[Lyubashenko, see \cite{Dr90}]\label{ex:Lyub}
Let $X$ be a non-empty set and let $f,g\colon X\to X$ be two bijections such that $fg=gf$. Define four binary operations: $x\circ y=f(y)$, $x\ld_\circ y=f^{-1}(y)$, $x\bullet y=g(x)$ and $x/_\bullet y=g^{-1}(x)$. Then $(X,\circ,\ld_{\circ},\bullet,/_{\bullet})$ is a birack called \emph{permutational}.
If $f=g=\id$, the birack is called a \emph{projection} one.
\end{example}

Conditions \eqref{eq:lq} and \eqref{eq:rq} mean that $(X,\circ,\ld_{\circ})$ is a \emph{left quasigroup} and $(X,\bullet,/_{\bullet})$ is a \emph{right quasigroup}.
Condition \eqref{eq:lq} simply means that
all \emph{left translations} $L_x\colon X\to X$ by $x$
\begin{align*}
L_x(a)= x\circ a,
\end{align*}
are bijections, with $L_x^{-1}(a)=x\ld_{\circ} a$. Equivalently, that for every $x,y\in X$, the equation $x\circ u=y$ has the unique solution $u=L_x^{-1}(y)$ in $X$. Similarly, Condition \eqref{eq:rq} gives that
all \emph{right translations} $\mathbf{R}_x\colon X\to X$ by $x$; $\mathbf{R}_x(a)= a\bullet x$,
are bijections with $\mathbf{R}_x^{-1}(a)=a/_{\bullet} x$.

\vskip 2mm

The \emph{left multiplication group} of a birack $(X,\circ,\ld_{\circ},\bullet,/_{\bullet})$ is the permutation group generated by left translations, i.e. the group $\LMlt(X)=\langle L_x:\ x\in X\rangle$. Similarly, one defines the \emph{right multiplication group} of $(X,\circ,\ld_{\circ},\bullet,/_{\bullet})$ as the permutation group generated by right translations, i.e. the group $\RMlt(X)=\langle \mathbf{R}_x:\ x\in X\rangle$. The permutation group $\Mlt(X)$ generated by all translations $L_x$ and $\mathbf{R}_x$ is called the \emph{multiplication group} of a birack.

We will say that a birack $(X,\circ,\ld_{\circ},\bullet,/_{\bullet})$ is \emph{left distributive}, if for every $x,y,z \in X$:
\begin{align}\label{eq:left}
x\circ(y\circ z)=(x\circ y)\circ(x\circ z)\quad \Leftrightarrow\quad L_xL_y=L_{x\circ y}L_x,
\end{align}
and it is \emph{right distributive}, if for every $x,y,z \in X$:
\begin{align}\label{eq:right}
(y\bullet z)\bullet x=(y\bullet x)\bullet(z\bullet x)\quad \Leftrightarrow\quad \mathbf{R}_x\mathbf{R}_z=\mathbf{R}_{z\bullet x}\mathbf{R}_x.
\end{align}
The birack is \emph{distributive} if it is left and right distributive. Each permutational birack is distributive.

\begin{remark}
A \emph{right rack} is a right distributive right quasigroup. These two conditions refer to second and third Reidemeister moves \cite{FR}. Biracks generalize racks in the following way. Let $(X,\circ,\ld_{\circ},\bullet,/_{\bullet})$ be a birack with all left translations being the identity permutation, i.e. for every $x,y, \in X$ one has
$$
x\circ y=y.
$$
Then $(X,\bullet,/_{\bullet})$ is a rack. 
\end{remark}
\vskip 2mm

A birack is {\em involutive} if it additionally satisfies, for every $x,y\in X$:
 \begin{align}
 &(x\circ y)\circ(x\bullet y)=x, \label{eq:linv}\\
&  (x\circ y)\bullet(x\bullet y)=y.\label{eq:rinv}
 \end{align}
By Condition \eqref{eq:linv}, operations $\bullet$ and $/_{\bullet}$ in an involutive birack are uniquely determined by operations $\circ$ and $\ld_{\circ}$ since $x\bullet y=(x\circ y)\ld_{\circ} x$. This allows to treat involutive biracks as left quasigroups satisfying additional conditions \cite[Proposition 1]{Rump}. In case of distributive involutive biracks one obtains the one-to-one correspondence to $2$-reductive racks \cite[Corollary 5.5]{JPZ19}.

In involutive biracks, $\LMlt(X)=\RMlt(X)=\Mlt(X)$. 
Moreover, by \cite[Corollary 5.8]{JPZ19} an involutive birack is left distributive if and only if it is right distributive.

\begin{example}
Let $(X,\circ,\ld_{\circ})$ be a left distributive left quasigroup (\emph{left rack}). Define operations $\bullet,/_{\bullet}\colon X\times X\to X$ as $x\bullet y=x=x/_{\bullet}y$. Then the structure $\mathbf{B}_L(X,\circ,\ld_{\circ})=(X,\circ,\ld_{\circ},\bullet,/_{\bullet})$ is a left distributive birack. Symmetrically, starting from a right rack $(Y,\triangleright,/_{\triangleright})$ and defining operations $\triangleleft,/_{\triangleleft}\colon Y\times Y\to Y$ as $x\triangleleft y=y=x\mathbin{/_{\triangleleft}}y$, one obtains a right distributive birack $\mathbf{B}_R(Y,\triangleright,/_{\triangleright})=(Y,\triangleleft,\ld_{\triangleleft},\triangleright,/_{\triangleright})$. We call such biracks \emph{left} and \emph{right derived biracks}, respectively. They are involutive only if they are projection ones.
\end{example}

\begin{example}
 Let $(X,\ast,\ld_{\ast})$ be a left rack and let $(Y,\triangle,/_\triangle)$ be a right rack. Then the product $\mathbf{B}_L(X,\ast,\ld_{\ast})\times \mathbf{B}_R(Y,\triangle,/_{\triangle})$ is a distributive birack with $\Mlt(X\times Y)\cong \LMlt(X)\times \RMlt(Y)$.
\end{example}

\begin{example}\label{ex:wada}
Let $(G,\cdot,e)$ be a group. Defining on the set $G$ binary operations as follows: $x\circ y=xy^{-1}x^{-1}$, $x\ld_{\circ} y=x^{-1}y^{-1}x$,  and $x\bullet y=xy^2$, $x/_{\bullet} y=xy^{-2}$, for $x,y\in G$, we obtain the birack $(G,\circ,\ld_{\circ},\bullet,/_{\bullet})$ known as the \emph{Wada switch} or \emph{Wada biquandle} (see \cite[Subsection 2.1(3)]{FJSK}).

Let $x,y,z\in G$. Direct calculations show that 
\begin{align*}
x\circ(y\circ z)&=xyzy^{-1}x^{-1},\\
(x\circ y)\circ (x\circ z)&=xy^{-1}zyx^{-1},
\end{align*} 
and the birack is left distributive if and only if $y^2z=zy^{2}$, for all $y,z\in G$, that means $y^2\in Z(G)$, for all $y\in G$. Furthermore,
\begin{align*}
(y\bullet z)\bullet x&=yz^2x^2, \qquad \text{ and}\\
(y\bullet x)\bullet(z\bullet x)&=yx^2zx^2zx^2.
\end{align*}
This implies that the birack is right distributive if and only if $x^2z=zx^{-2}$, for all $x,z\in G$,
which is equivalent to $x^4=e$ and $x^2\in Z(G)$, for all $x\in G$.
Thus, if the birack is right distributive then it is left distributive as well.

Moreover,
\begin{align*}
(x\circ y)\circ(x\bullet y)&=x(y^{-1}x^{-1})^2, \qquad \text{ and}\\
(x\circ y)\bullet(x\bullet y)&= xyxy^2.
 \end{align*}
Hence, by \eqref{eq:linv} and \eqref{eq:rinv}, the birack is involutive if and only if $(xy)^2=e$, for all $x,y\in G$, that means if $(G,\cdot,e)$ is an elementary abelian $2$-group.

For instance, there are five groups of order~$8$.
One of them is cyclic of exponent~$8$ and therefore
its Wada switch is not distributive. One of them is elementary abelian and its Wada switch is a projection birack. The other three groups (namely $\Z_4\times \Z_2$, $D_8$ and $Q_8$) are of exponent~$4$ and all their square elements fall within the
centers and therefore these groups yield non-involutive distributive
biracks.

\end{example}

\begin{example}\label{ex:wada_ab}
Let $(G,+,0)$ be an abelian group. Then, the birack operations defined in Example \ref{ex:wada} look as follows: $x\circ y=x\ld _{\circ}y=-y$ and $x\bullet y=x+2y$, $x/_{\bullet} y=x-2y$, for $x,y\in G$. Clearly, such a birack is always left distributive. It is a non-involutive distributive
birack if and only if $G$ is an abelian group of exponent
exactly~$4$.
\end{example}

\begin{lemma}\label{lem:ld_equiv}
Let $(X,\circ,\ld_{\circ},\bullet,/_{\bullet})$ be a birack. The following are equivalent:
\begin{enumerate}
\item [(i)] $(X,\circ,\ld_{\circ},\bullet,/_{\bullet})$ is left distributive;
\item [(ii)] $(X,\circ,\ld_{\circ},\bullet,/_{\bullet})$ satisfies, for every $x,y\in X$,
\begin{equation}\label{eq:birld}
L_x=L_{x\bullet y}=L_{\mathbf{R}_y(x)};
\end{equation}
\item [(iii)] $(X,\circ,\ld_{\circ},\bullet,/_{\bullet})$ satisfies, for every $x,y\in X$,
\begin{equation}\label{eq:birrd}
L_x=L_{x/_{\bullet} y}=L_{\mathbf{R}^{-1}_y(x)},
\end{equation}
\item [(iv)] Left translations by elements taken from the same orbit of the action of the group $\RMlt(X)$ on a set $X$ are equal permutations on $X$.
\end{enumerate}
\end{lemma}
\begin{proof}
Indeed, by \eqref{eq:b1} and \eqref{eq:lq}, we have for $x,y,z\in X$
\begin{align*}
&(x\circ y)\circ(x\circ z)=x\circ(y\circ z)\quad \Leftrightarrow\quad
(x\circ y)\circ(x\circ z)=(x\circ y)\circ((x\bullet y)\circ z)\quad \Leftrightarrow\\
&x\circ z=(x\bullet y)\circ z\; \Leftrightarrow\; L_x=L_{x\bullet y}.
\end{align*}

Additionally, by \eqref{eq:rq}, substituting of $x$ by $x/_{\bullet}y$ in \eqref{eq:birld} we immediately obtain
\[
L_x=L_{x/_{\bullet}y}.
\]
Similarly, substituting of $x$ by $x\bullet y$ in \eqref{eq:birrd} we have
\[
L_{x\bullet y}=L_x.
\]
Finally, (ii) $\Leftrightarrow$ (iv) follows by the fact that for any $x\in X$ its orbit of the action $\RMlt(X)$ on $X$ consists exactly of elements $\alpha(x)$ for $\alpha\in \RMlt(X)$.
\end{proof}

Analogously, due to \eqref{eq:b3} and \eqref{eq:rq}, a birack is right distributive if and only if
\begin{equation}\label{eq:bilrd}
\mathbf{R}_x=\mathbf{R}_{y\circ x}=\mathbf{R}_{L_y(x)},
\end{equation}
or equivalently, right translations by elements taken from the same orbit of the action of the group $\LMlt(X)$ on $X$ are equal permutations on $X$.

\vskip 2mm

By results of \cite[Section 3]{JPZ19} left (right) distributivity in involutive biracks is equivalent to commutativity of the left (right) multiplication group. For a non-involutive distributive birack it is not always true (see Example \ref{ex:6}). But even then left and right multiplication groups commute.

\begin{lemma}\label{lem:com_lr}
 Let $(X,\circ,\ld_{\circ},\bullet,/_{\bullet})$ be a distributive birack. Then,
\begin{equation}\label{eq:com_lr}
 [\LMlt(X),\RMlt(X)]=\{\id\}.
\end{equation}
\end{lemma}

\begin{proof}
For $x,y,z\in X$ one has,
\begin{align*}
  &L_x\mathbf{R}_y(z)\eqrel{eq:birld}{=} L_{x\bullet (z\circ y)} \mathbf{R}_y(z)
  = (x\bullet(z\circ y))\circ(z\bullet y)\eqrel{eq:b2}=\\
  &(x\circ z)\bullet((x\bullet z)\circ y) =
  \mathbf{R}_{(x\bullet z)\circ y}L_x(z) \eqrel{eq:bilrd}=
  \mathbf{R}_yL_x(z).
	\end{align*}
\end{proof}

Lemmas \ref{lem:ld_equiv} and \ref{lem:com_lr} allow us to characterize distributive biracks in an alternative way.

\begin{proposition}\label{alt_def}
 Let $(X,\circ,\ld_{\circ},\bullet,/_{\bullet})$
 be a structure with four binary operations. Then
 $(X,\circ,\ld_{\circ},\bullet,/_{\bullet})$ is a distributive birack if and only if the following conditions are satisfied
 \begin{enumerate}
  \item [(i)] $(X,\circ,\ld_\circ)$ is a left rack
  and $(X,\bullet,/_\bullet)$ is a right rack,
  \item [(ii)] $(X,\circ,\ld_{\circ},\bullet,/_{\bullet})$ satisfies \eqref{eq:birld}, \eqref{eq:bilrd}
  and \eqref{eq:com_lr}.
 \end{enumerate}
\end{proposition}
\begin{proof}
 If $(X,\circ,\ld_\circ)$ is a left rack then \eqref{eq:b1} and \eqref{eq:birld} are equivalent,
 as we showed in Lemma~\ref{lem:ld_equiv}. Analogously,
 \eqref{eq:b3} and \eqref{eq:bilrd} are equivalent
 when $(X,\bullet,/_\bullet)$ is a right rack.
 Finally, when supposing \eqref{eq:bilrd} and \eqref{eq:birld}, Conditions \eqref{eq:b2} and \eqref{eq:com_lr} are equivalent, as we saw in the proof
 of Lemma~\ref{lem:com_lr}.
\end{proof}
\begin{remark}
Condition $\text{(ii)}$ in Proposition \ref{alt_def} can be formulated in the equivalent way as follows: for all $x,y,z\in X$
	\begin{align*}
	&(x\bullet y)\circ z=x\circ z,\\
	&x\bullet (y\circ z)=x\bullet z,\\
	&x\circ (y\bullet z)=(x\circ y)\bullet z.
	\end{align*}
\end{remark}

Obviously, the left distributivity of  $(X,\circ,\ld_{\circ},\bullet,/_{\bullet})$ means that all left translations are automorphisms of  $(X,\circ,\ld_{\circ})$.
Additionally, directly from \eqref{eq:left}
we obtain that the left distributivity implies,
 for every $x,y\in X$,
\begin{align}\label{eq:ab}
L_{x\circ y}=L_xL_{y}L_x^{-1} \quad {\rm and}\quad L_{x\ld_{\circ} y}=L_x^{-1}L_{y}L_x.
\end{align}
Note also that, for an arbitrary automorphism $\alpha$ of $(X,\circ,\ld_{\circ})$, we have
\begin{align}\label{eq:ab1}
L_{\alpha(x)}(y)=\alpha(x)\circ y=\alpha(x\circ \alpha^{-1}(y))=\alpha L_{x}\alpha^{-1}(y).
\end{align}

Similarly, for a right distributive birack $(X,\circ,\ld_{\circ},\bullet,/_{\bullet})$, we have
\begin{align}\label{eq:ab2}
\mathbf{R}_{x\bullet y}=\mathbf{R}_y\mathbf{R}_{x}\mathbf{R}_y^{-1} \quad {\rm and}\quad \mathbf{R}_{x/_\bullet y}=\mathbf{R}_y^{-1}\mathbf{R}_{x}\mathbf{R}_y.
\end{align}
Moreover, for an arbitrary automorphism $\beta$ of $(X,\bullet,/_\bullet)$, we have
\begin{align}\label{eq:ab3}
\mathbf{R}_{\beta(x)}(y)=y\bullet\beta(x)=\beta(\beta^{-1}(y)\bullet x)=\beta \mathbf{R}_{x}\beta^{-1}(y).
\end{align}

\begin{example}\label{ex:6}
Let $(X,\circ,\ld_\circ,\bullet,/_\bullet)$ be the following
structure: $X=\{1,2,3,4,5,6\}$ and
\begin{align*}
 L_1&=(3 5 4 6) & L_2&=(6 4 5 3) \\
 L_3=L_4&=(1 2)(5 6) & L_5=L_6&=(1 2)(3 4)\\
 \mathbf{R}_1=\mathbf{R}_2&=\mathrm{id} &
 \mathbf{R}_3=\mathbf{R}_4= \mathbf{R}_5=\mathbf{R}_6&=(3 4)(5 6)
\end{align*}
By hand, or using a GAP library named `RiG' \cite{RIG}, we can show that the automorphism group of~$(X,\circ,\ld_{\circ},\bullet,/_{\bullet})$,
is the group generated by the permutations $L_1$, $L_3$ and $L_5$.
We can now easily prove that this structure is a distributive birack.
Indeed, all $L_x$, for $x\in X$, are automorphims of $(X,\circ,\ld_\circ)$,
as well as all $\mathbf{R}_x$, for $x\in X$, are automorphisms of $(X,\bullet,/_\bullet)$
and therefore $(X,\circ,\ld_\circ,\bullet,/_\bullet)$ is both left and right distributive birack.
  The group $\LMlt(X)$ is  a non-abelian group of order~$8$
  having two orbits, namely $\{1,2\}$ and $\{3,4,5,6\}$.
  Condition~\eqref{eq:birld} is satisfied
  since $\mathbf{R}_1=\mathbf{R}_2$ and $\mathbf{R}_3=\mathbf{R}_4= \mathbf{R}_5=\mathbf{R}_6$.
  The group $\RMlt(X)$ is a two-element group with four orbits
  $\{1\}$, $\{2\}$, $\{3,4\}$ and $\{5,6\}$.
  Condition~\eqref{eq:bilrd} is satisfied
  since $L_3=L_4$ and $L_5=L_6$.
  The group $\RMlt(X)$ is equal to the center of $\LMlt(X)$ and we have therefore Condition~\eqref{eq:com_lr}.
  We may also notice that this example is not a Wada biquandle (Example~\ref{ex:wada}) since there exists no $6$-element group of exponent~$4$.
\end{example}

\begin{proposition}\label{prop:aut}
 Let $(X,\circ,\ld_{\circ},\bullet,/_{\bullet})$
 be a distributive birack. Then, for each $x\in X$,
 the bijections $L_x$ and $\mathbf{R}_x$ are
 automorphisms of~$(X,\circ,\ld_{\circ},\bullet,/_{\bullet})$.
\end{proposition}

\begin{proof}
 The property $L_x(y\circ z)=L_x(y)\circ L_x(z)$
 is the definition of left distributivity. Substituting $z\mapsto y\ld_\circ u$ we obtain $L_x(u)=L_x(y)\circ L_x(y\ld_\circ u)$ from where we get
\begin{align}\label{eq:aut}
 L_x(y) \ \ld_\circ \ L_x(u) = L_x(y\ld_\circ u).
\end{align}
Now
 \begin{align*}
  L_x(y\bullet z) &= L_x \mathbf{R}_z (y) \eqrel{eq:com_lr}=
  \mathbf{R}_z L_x(y) \eqrel{eq:bilrd}= \mathbf{R}_{L_x(z)} L_x(y) = L_x(y) \bullet L_x(z),\\
  L_x(y/_\bullet z) &= L_x \mathbf{R}_z^{-1} (y) \eqrel{eq:com_lr}=
  \mathbf{R}_z^{-1} L_x(y) \eqrel{eq:bilrd}= \mathbf{R}_{L_x(z)}^{-1} L_x(y) = L_x(y) /_\bullet L_x(z),
 \end{align*}
  and therefore $L_x$ is a homomorphism of $(X,\bullet,/_\bullet)$. The claim
  for $\mathbf{R}_x$ is proved analogously.
\end{proof}

By Proposition \ref{prop:aut} and Conditions \eqref{eq:ab1}, \eqref{eq:ab3} one immediately obtains that groups $\LMlt(X)$, $\RMlt(X)$, $\Mlt(X)$ are normal subgroups of the automorphism group of a distributive birack\\ $(X,\circ,\ld_{\circ},\bullet,/_{\bullet})$.

\section{Multipermutational biracks}\label{sec:MPB}
Gateva-Ivanova characterized involutive multipermutation solutions of the Yang-Baxter equation in the language of some identities satisfied by corresponding biracks, see e.g. \cite[Proposition 4.7]{GI18}. We will generalize her result for (non-involutive) left distributive case. Let us start with some auxiliary definitions.

\begin{de}
Let $m\in \mathbb{N}$. A birack $(X,\circ,\ld_{\circ},\bullet,/_{\bullet})$ is called:
\begin{enumerate}
\item [(i)] \emph{idempotent}, if for every $x\in X$:
\begin{align}\label{eq:idemp}
x\circ x=x=x\bullet x
\end{align}
\item [(ii)] \emph{left}  $m$-\emph{reductive},  if for every $x_0,x_1,x_2,\ldots,x_m\in X$:
\begin{align}\label{eq:mred}
(\ldots((x_0\circ x_1)\circ x_2)\ldots)\circ x_m=(\ldots((x_1\circ x_2)\circ x_3)\ldots)\circ x_m
\end{align}
\item [(iii)] \emph{right} $m$-\emph{reductive},  if for every $x_0,x_1,x_2,\ldots,x_m\in X$:
\begin{align}\label{eq:rmred}
x_0\bullet(\ldots(x_{m-2}\bullet(x_{m-1}\bullet x_m))\ldots )=x_0\bullet(\ldots(x_{m-3}\bullet(x_{m-2}\bullet x_{m-1}))\ldots )
\end{align}
\item [(iv)] \emph{left} $m$-\emph{permutational}, if for every $x,y,x_1,x_2,\ldots,x_m\in X$:
\begin{align}\label{eq:mper}
(\ldots((x\circ x_1)\circ x_2)\ldots)\circ x_m=(\ldots((y\circ x_1)\circ x_2)\ldots)\circ x_m
\end{align}
\item [(v)] \emph{right} $m$-\emph{permutational}, if for every $x,y,x_1,x_2,\ldots,x_m\in X$:
\begin{align}\label{eq:rmper}
x_0\bullet(\ldots(x_{m-2}\bullet(x_{m-1}\bullet x))\ldots )=x_0\bullet(\ldots(x_{m-2}\bullet(x_{m-1}\bullet y))\ldots ).
\end{align}
\end{enumerate}
\end{de}

\noindent A birack is $m$-\emph{reductive} ($m$-\emph{permutational}) if it is both left and right $m$-reductive ($m$-permutational).

\begin{example}
 A distributive birack from Example~\ref{ex:wada} is right $2$-reductive because $x\bullet (y \bullet z)=x(yz^2)^2=xyz^2yz^2=xy^2=x\bullet y$, since $z^2\in Z(G)$ and $z^4=e$.
\end{example}

\begin{lemma}\label{lem:red-perm}
 Let $(X,\circ,\ld_{\circ},\bullet,/_{\bullet})$ be a left $m$-permutational birack.
 \begin{itemize}
  \item[(i)] If $(X,\circ,\ld_{\circ},\bullet,/_{\bullet})$ is idempotent then it is left $m$-reductive.
  \item[(ii)] If $(X,\circ,\ld_{\circ},\bullet,/_{\bullet})$ is left distributive and $m\geq 2$ then
  it is left $m$-reductive.
 \end{itemize}
\end{lemma}
\begin{proof}
 (i) is evident. For (ii) we have
 \begin{align*}
  &(\cdots((x_0 \circ x_1)\circ x_2)\cdots)\circ x_m
  \eqrel{eq:mper}= (\cdots((x_1 \circ x_1)\circ x_2)\cdots)\circ x_m
  =\\
	&(\cdots((L_{x_1\circ x_1} (x_2))\cdots)\circ x_m
  \eqrel{eq:ab}= (\cdots((L_{x_1} (x_2))\cdots)\circ x_m
  = (\cdots(x_1\circ x_2)\cdots)\circ x_m,
 \end{align*}
for every $x_0,x_1,x_2,\ldots,x_m\in X$.
\end{proof}

By Example \ref{ex:Lyub} there exists
left $1$-permutational left distributive birack that is not left $1$-reductive. It is a permutational birack with $f\neq\id$.

\vskip 3mm

Let $(X,\circ,\ld_{\circ},\bullet,/_{\bullet})$ be a birack. Etingof, Schedler and Soloviev defined in \cite[Section 3.2]{ESS} the relation
\begin{equation}\label{rel:sim}
a\sim b \quad \Leftrightarrow\quad L_a=L_b \quad \Leftrightarrow\quad \forall x\in X\quad a\circ x=b\circ x.
\end{equation}

By their results, the relation $\sim$ is a \emph{congruence} of an involutive birack, i.e. an equivalence relation on the set $X$ compatible with all four operations in a birack $(X,\circ,\ld_{\circ},\bullet,/_{\bullet})$. For a detailed definition see \cite[Definition 3.1]{JPZ}.

In the case of non-involutive biracks,
the equivalence~$\sim$ need not be a congruence
(see \cite[Example 3.4]{JPZ})
but it is so if the birack is left distributive.
\begin{theorem}\label{th:left_congr}
Let $(X,\circ,\ld_{\circ},\bullet,/_{\bullet})$ be a left distributive birack. Then the relation \eqref{rel:sim} is a congruence of
$(X,\circ,\ld_{\circ},\bullet,/_{\bullet})$.
\end{theorem}
\begin{proof}
By \eqref{eq:lq}, \eqref{eq:ab} and \eqref{eq:birld} the proof is straightforward. Let $a\sim x$ and $b\sim y$. Then
\begin{align*}
&L_{a\circ b} \stackrel{\scriptsize \eqref{eq:ab}}= L_aL_bL^{-1}_a =L_xL_yL^{-1}_x=L_{x\circ y}\quad \Rightarrow\quad a\circ b\sim x\circ y,\\
&L_{a\ld_{\circ} b}\stackrel{\scriptsize \eqref{eq:ab}}= L_a^{-1}L_bL_a=L_x^{-1}L_yL_x=L_{x\ld_{\circ} y}\quad \Rightarrow\quad a\ld_{\circ} b\sim x\ld_{\circ} y,\\
&L_{a\bullet b}\stackrel{\scriptsize \eqref{eq:birld}}=L_a=L_x=L_{x\bullet y}\quad \Rightarrow\quad a\bullet b\sim x\bullet y,\\
& L_{a/_{\bullet} b}\stackrel{\scriptsize \eqref{eq:birld}}=L_{(a/_{\bullet} b)\bullet b}
\stackrel{\scriptsize \eqref{eq:lq}}=L_a=L_x=L_{(x/_{\bullet} y)\bullet y}=L_{x/_{\bullet} y}\quad \Rightarrow\quad \mathord{a/_{\bullet} b}\sim \mathord{x/_{\bullet} y}.\qedhere
\end{align*}
\end{proof}

\begin{lemma}\label{lm:sf}
Let $(X,\circ,\ld_{\circ},\bullet,/_{\bullet})$ be a left distributive birack. Then the quotient birack\\
$(X/\mathord{\sim},\circ,\ld_{\circ},\bullet,/_{\bullet})$ is idempotent and equal to the left derived birack $\mathbf{B}_L(X/\mathord{\sim},\circ,\ld_{\circ})$.
\begin{proof}
By the left distributivity and \eqref{eq:birld}, for every $x,y\in X$,
\[
x\sim x\bullet y.
\]
Furthermore, by \eqref{eq:ab}, $L_x=L_xL_xL_x^{-1}=L_{x\circ x}$,
which shows that $x\sim x\circ x$ and $(X/\mathord{\sim},\circ,\ld_{\circ},\bullet,/_{\bullet})$ is idempotent.
\end{proof}
\end{lemma}

Analogously to \eqref{rel:sim}, we can define symmetrical relation
\begin{equation}
a\backsim b \quad \Leftrightarrow\quad \mathbf{R}_a=\mathbf{R}_b \quad \Leftrightarrow\quad \forall x\in X\quad x\bullet a= x\bullet b
\end{equation}
and this relation is a congruence of every right distributive birack. If a birack is involutive then $a\sim b$ if and only if $a\backsim b$ \cite[Proposition 2.2]{ESS}.

\begin{de}
Let $(X,\circ,\ld_{\circ},\bullet,/_{\bullet})$ be a left distributive birack. The left derived birack $\mathbf{B}_L(X/\mathord{\sim},\circ,\ld_{\circ})$ is called {\em left retract} of~$X$ and denoted by $\mathrm{LRet}(X)$. One defines \emph{iterated left retraction} in the following way: ${\rm LRet}^0(X)=(X,\circ,\ld_{\circ},\bullet,/_{\bullet})$ and
${\rm LRet}^k(X)={\rm LRet}({\rm LRet}^{k-1}(X))$, for any natural number $k>1$.
\end{de}

The {\em right retract} and \emph{iterated right retraction} are defined analogously.

\begin{remark}\label{rem:join}
Let $(X,\circ,\ld_{\circ},\bullet,/_{\bullet})$ be a distributive birack and let $\Theta$ be the join of the congruences $\sim$ and $\backsim$ (the least congruence containing both of them). Then the quotient birack $(X/\mathord{\Theta},\circ,\ld_{\circ},\bullet,/_{\bullet})$ is the projection one.
\end{remark}

The intersection of the two relations here defined is
the relation
\begin{equation}
 a\approx b \quad \Leftrightarrow \quad {a\sim b} \wedge {a\backsim b}
 \quad \Leftrightarrow \quad {L_a=L_b} \wedge {\mathbf{R}_a=\mathbf{R}_b}.
\end{equation}
When a birack is distributive then this equivalence
is an intersection of two congruences and therefore a congruence. Nevertheless, it is a congruence even in the case of
general biracks, however the proof is rather complicated
and technical~\cite[Theorem 3.3]{JPZ}.

\vskip 2mm

 Let $(X,\circ,\ld_{\circ},\bullet,/_{\bullet})$ be a birack. The {\em retract} of~$X$,
 denoted by $\mathrm{Ret}(X)$, is
 the quotient birack $(X/\mathord{\approx},\circ,\ld_{\circ},\bullet,/_{\bullet})$. Similarly, as for the congruences $\sim$ and $\backsim$, one defines \emph{iterated retraction} as ${\rm Ret}^0(X)=(X,\circ,\ld_{\circ},\bullet,/_{\bullet})$ and
${\rm Ret}^k(X)={\rm Ret}({\rm Ret}^{k-1}(X))$, for any natural number $k>1$.

In the case of involutive biracks all three notions of
retracts coincide.

\begin{corollary}\label{cor:sf1}
Let $(X,\circ,\ld_{\circ},\bullet,/_{\bullet})$ be a distributive birack. Then $\mathrm{Ret}(X)$ is an idempotent birack.
\end{corollary}

\begin{example}\label{ex:quotient}
Let $(G,\circ,\ld_{\circ},\bullet,/_{\bullet})$ be the distributive birack from Example \ref{ex:wada}. It is easy to see that, for any $x,y\in G$, 
\begin{align*}
&x\sim y\quad \Leftrightarrow \quad   xy^{-1}\in Z(G).
 \end{align*}
and 
\begin{align*}
&x\backsim y\quad \Leftrightarrow \quad   x^2=y^2.
 \end{align*}
Note that the relation $\approx$ is different than equality relation if and only if $Z(G)$ contains at least one element of order $2$.
It is also easy to see that $\sim$ is a subrelation of $\backsim$ if and only if $Z(G)$ is an elementary abelian $2$-group.

For instance, in the quaternion group~$Q_8$,
$\sim$ and $\backsim$ are different and
the relation $\approx$ is equal to the conjugation relation. Thus $(Q_8/\approx,\circ,\ld_{\circ},\bullet,/_{\bullet})$ is a $4$-element projection birack.

For an abelian group $(G,+,0)$ satisfying the condition $4x=0$, clearly the relations $\approx$ and $\backsim$ coincide. In this case, the birack $(G/\mathord{\approx},\circ,\ld_{\circ},\bullet,/_{\bullet})$ is a projection birack.
\end{example}

\section{Nilpotent permutation group}\label{sec:PG}
A birack is of {\em multipermutation level}~$k$, if $|\mathrm{Ret}^k(X)|=1$ and $|\mathrm{Ret}^{k-1}(X)|>1$.

This means that applying $k$ times the congruence $\mathord{\approx}$ to the subsequent quotient biracks, one obtains the one-element birack.

In \cite[Proposition 4.7]{GI18} Gateva-Ivanova proved that an involutive birack $(X,\circ,\ld_{\circ},\bullet,/_{\bullet})$
is of multipermutation level~$k$ if and only if it is left $k$-permutational. The very same proof works for (non-involutive) left distributive biracks. In the distributive case we have an additional equivalent condition (nilpotency of the left multiplication group). Therefore, we decided to give a different proof
that uses this condition.

Let $G$ be a group. We recall the definition of \emph{lower central series of the group $G$}, for $i\in \mathbb{N}$:
\begin{align*}
&\gamma_0(G)=G,\\
&\gamma_{i}(G)=[\gamma_{i-1}(G),G]=[G,\gamma_{i-1}(G)], \text{ for }i\geq 1.
\end{align*}

Then $G$ is \emph{nilpotent of class} $k$, if $k$ is the smallest number for which $\gamma_{k}(G)=\{1\}$.
In particular, $G$ is nilpotent of class $0$ if and only if $G$ is a trivial group.
If $k>0$, nilpotency of class~$k$ is equivalent to the property that $G/Z(G)$ is nilpotent of class $k-1$.

\begin{theorem}\label{thm:ld_nilp_i}
 Let $(X,\circ,\ld_{\circ},\bullet,/_{\bullet})$ be a left distributive idempotent birack and let $k\geq 1$. Then
 the following conditions are equivalent:
 \begin{enumerate}
  \item[(i)] $|\mathrm{LRet}^k(X)|=1$,
  \item[(ii)] $(X,\circ,\ld_{\circ},\bullet,/_{\bullet})$ is left $k$-reductive,
  \item[(iii)] $(X,\circ,\ld_{\circ},\bullet,/_{\bullet})$ is left $k$-permutational,
  \item[(iv)] $\LMlt(X)$ is nilpotent of class at most~$k-1$.
 \end{enumerate}
\end{theorem}

\begin{proof}
 (ii) $\Leftrightarrow$ (iii) is Lemma~\ref{lem:red-perm}.

 (ii)$\Leftrightarrow$(iv)
 
 Let $k>1$.
 We translate the notion of $k$-reductivity in the language of permutations.
 Let $(x_0,\ldots,x_k)\in X^{k+1}$ and we denote this sequence by $\chi$.
 We write, by induction, $y_{0,\chi}=x_0$ and $y_{i+1,\chi}=y_{i,\chi} \circ x_{i+1}$
 as well as $z_{1,\chi}=x_1$ and $z_{i+1,\chi}=z_{i,\chi} \circ x_{i+1}$.
 The equation of left $k$-reductivity then is $y_{k,\chi}=z_{k,\chi}$.

Let us write $\alpha_{-1,\chi}=\mathrm{id}$ and $\alpha_{i+1,\chi}=\alpha_{i,\chi} L_{x_{i+1}}\alpha_{i,\chi}^{-1}$.
 We prove by induction that $L_{y_{i,\chi}}=\alpha_{i,\chi}$, for $0\leq i< k$.
 The claim is clear for $i=0$ and
\begin{align*}
L_{y_{i+1,\chi}}=L_{y_{i,\chi} \circ x_{i+1}}=L_{L_{y_{i,\chi}}(x_{i+1})}=L_{\alpha_{i,\chi}(x_{i+1})}
 =\alpha_{i,\chi} L_{x_{i+1}}\alpha_{i,\chi}^{-1}=\alpha_{i+1,\chi}.
\end{align*}
 Analogously, we write $\beta_{0,\chi}=\mathrm{id}$ and $\beta_{i+1,\chi}=\beta_{i,\chi} L_{x_{i+1}}\beta_{i,\chi}^{-1}$
 and we get $L_{z_{i,\chi}}=\beta_{i,\chi}$.
 The left $k$-reductivity is, whenever $k>1$,
\begin{align*}
 y_{k,\chi}=z_{k,\chi} \quad \Leftrightarrow \quad y_{k-1,\chi}\circ x_k=z_{k-1,\chi}\circ x_k \quad \Leftrightarrow \quad L_{y_{k-1,\chi}}=L_{z_{k-1,\chi}} \quad \Leftrightarrow \quad \alpha_{k-1,\chi}=\beta_{k-1,\chi},
\end{align*}
  for all $\chi\in X^{k+1}$.

We now prove by induction that the group $\gamma_j(\LMlt(X))$ is generated by the set $\{\alpha_{j,\chi}^{-1}\beta_{j,\chi}\colon \chi\in X^{k+1}\}$. The claim is clear for $j=0$ since $\{\alpha_{0,\chi}^{-1}\beta_{0,\chi}=L_{x_0}^{-1}\colon x_0\in X\}$.

Let $j>0$. Then, by the induction hypothesis,
\begin{align*}
&\alpha_{j+1,\chi}^{-1}\beta_{j+1,\chi}=\alpha_{j,\chi} L_{x_{j+1}}^{-1}\alpha_{j,\chi}^{-1}\beta_{j,\chi} L_{x_{j+1}}\beta_{j,\chi}^{-1}=\alpha_{j,\chi} L_{x_{j+1}}^{-1}(\beta_{j,\chi}^{-1}\alpha_{j,\chi})^{-1} L_{x_{j+1}}(\beta_{j,\chi}^{-1}\alpha_{j,\chi})\alpha_{j,\chi}^{-1}=\\
&[L_{x_{j+1}},\beta_{j,\chi}^{-1}\alpha_{j,\chi}]^{\alpha_{j,\chi}^{-1}}\in [\LMlt(X),\gamma_j(\LMlt(X))]=\gamma_{j+1}(\LMlt(X)),
\end{align*}
which shows that $\left\langle \alpha_{j+1,\chi}^{-1}\beta_{j+1,\chi}\colon \chi\in X^{k+1}\right\rangle\subseteq \gamma_{j+1}(\LMlt(X))$.

Take 
$y\in X$ and $\alpha_{j,\chi}$, $\beta_{j,\chi}$ for some $\chi=(x_0,\ldots ,x_j,\ldots ,x_k)$. Consider the new sequence $\zeta=(x_0,\ldots ,x_j,y,x_{j+2},\ldots ,x_k)$. Obviously, $\beta_{j,\chi}^{-1}\alpha_{j,\chi}=\beta_{j,\zeta}^{-1}\alpha_{j,\zeta}$. Hence,
\begin{align*}
[L_{y},\beta_{j,\zeta}^{-1}\alpha_{j,\zeta}]=\alpha_{j,\zeta}\alpha_{j+1,\zeta}^{-1}\beta_{j+1,\zeta}\alpha_{j,\zeta}^{-1},
\end{align*}
 which shows the inverse inclusion.

We are nearing the final argument. A birack $(X,\circ,\ld_{\circ},\bullet,/_{\bullet})$ is left $k$-reductive, for $k\geq 2$,
if and only if $\alpha_{k-1,\chi}=\beta_{k-1,\chi}$ for each choice of the sequence $\chi$. This is equivalent to the fact that $\alpha_{k-1,\chi}^{-1}\beta_{k-1,\chi}=\id$ for each $\chi\in X^{k+1}$ or to the fact $\gamma_{k-1}(\LMlt(X))=\{\id\}$ and this is equivalent to $\LMlt(X)$ being nilpotent of class at most $k-1$.

For $k=1$, the group $\LMlt(X)$ is nilpotent of class~$0$ if and only if it is trivial which is clearly equivalent to $(X,\circ,\ld_{\circ},\bullet,/_{\bullet})$ being $1$-reductive, which completes the proof.

(i)$\Leftrightarrow$(iv)
It is clear that the structure of the left retract $\mathrm{LRet}(X)$ may be formally defined as the birack $(\tilde X=\{L_x\colon x\in X\},\tilde \circ,\ld_{\tilde \circ},\tilde\bullet,/_{\tilde\bullet})$ such that 
\begin{align*}
&L_x \mathbin{\tilde \circ}L_y=L_{x\circ y}\stackrel{\rm def}{=}\mathcal{L}_{L_x}(L_y),\\
&L_x \mathbin{\ld_{\tilde \circ}}L_y=L_{x\ld_\circ y},\\
&L_x \mathbin{\tilde \bullet}L_y=L_{x}= L_x \mathbin{/_{\tilde\bullet}}L_y.
\end{align*}
It is so since the mapping $\kappa\colon \mathord{X/\mathord{\sim}} \to \tilde X$, $\mathord{x/\mathord{\sim}} \mapsto L_x$ is a well-defined isomorphism of biracks.

We define the following mapping $\Phi\colon\LMlt(X)\to\LMlt(\tilde X)$: 
$$\Phi(\alpha)(L_x)=\alpha L_x\alpha^{-1}=L_{\alpha(x)}.$$
The mapping $\Phi$ is onto since $\Phi(L_y)=\mathcal{L}_{L_y}$. And it is a homomorphism since
$$\Phi(\alpha\beta)(L_x)=L_{\alpha\beta(x)}\eqrel{eq:ab1}=\Phi(\alpha)(L_{\beta(x)})\eqrel{eq:ab1}=\Phi(\alpha)\Phi(\beta)(L_x).$$
Now we compute the kernel of the homomorphism:
$$\Phi(\alpha)=\mathrm{id} \Leftrightarrow \Phi(\alpha)(L_x)=L_x \Leftrightarrow \alpha L_x\alpha^{-1}=L_x
\Leftrightarrow \alpha\in Z(\LMlt(X)).$$
Hence $\LMlt(X)$ is nilpotent of class~$k$ if and only if $\LMlt(\mathrm{LRet}(X))$ is nilpotent
of class $k-1$. We finish the proof by noticing that
$|\mathrm{LRet}(X)|=1$ if and only if
$\LMlt(X)$ is nilpotent of class~$0$.
\end{proof}

The same theorem is not true for non-idempotent
biracks. Non-idempotency of permutational birack is equivalent to the fact that $f\neq\id$. This means that
$|\mathrm{LRet}(X)|=1$ but the left multiplication group $\LMlt(X)$ is non-trivial.
However, this is the only exception.

\begin{theorem}\label{thm:ld_nilp_ni}
 Let $(X,\circ,\ld_{\circ},\bullet,/_{\bullet})$ be a left distributive birack and let $k\geq 2$. Then
 the following conditions are equivalent:
 \begin{enumerate}
  \item[(i)] $|\mathrm{LRet}^k(X)|=1$,
  \item[(ii)] $(X,\circ,\ld_{\circ},\bullet,/_{\bullet})$ is left $k$-reductive,
  \item[(iii)] $(X,\circ,\ld_{\circ},\bullet,/_{\bullet})$ is left $k$-permutational,
  \item[(iv)] $\LMlt(X)$ is nilpotent of class at most~$k-1$.
 \end{enumerate}
\end{theorem}

\begin{proof}
 (ii)$\Leftrightarrow$(iii) is Lemma~\ref{lem:red-perm}.
 The equivalence (ii)$\Leftrightarrow$(iv) was proved
 in the proof of Theorem~\ref{thm:ld_nilp_i}
 as the idempotency was not used in this part of the proof.
 The only place
 where the idempotency was used was actually the case of $k=1$ in
 (i)$\Leftrightarrow$(iv).
 Hence we can reason again that $\LMlt(X)$ is nilpotent
 of class $k$ if and only if $\LMlt(\mathrm{LRet}(X))$
 is nilpotent of class $k-1$. According to Lemma~\ref{lm:sf},
 $\mathrm{LRet}(X)$ is idempotent and therefore,
 according to Theorem~\ref{thm:ld_nilp_i},
 $\LMlt(\mathrm{LRet}(X))$
 is nilpotent of class $k-1$ if and only if $|\mathrm{LRet}^{k+1}(X)|=1$ and $|\mathrm{LRet}^{k}(X)|>1$.
\end{proof}

Since the operation~$\bullet$ was never used
in the proof, we immediately get:

\begin{corollary}\label{cor:si}
 Let $(X,\circ,\ld_\circ)$ be a left rack. Then $(X,\circ,\ld_\circ)$ is $k$-reductive
 if and only if $\LMlt(X)$ is nilpotent of class at most~$k-1$.
\end{corollary}

Corollary \ref{cor:si} generalizes the result obtained by the authors for \emph{medial quandles} (a proper subclass of idempotent left racks) \cite[Theorem 5.3]{JPZ18}. The proof given there used different methods.

\begin{remark}
It is worth emphasizing that all results established for \emph{left} properties (distributivity, $m$-reductivity, $m$-permutationality, retracts) are also true for right ones, when using their dual versions.
\end{remark}

\begin{lemma}\label{lem:nilp_com}
 Let $G$ be a group and $H_1,H_2$ be its subgroups such that $[H_1,H_2]=\{1\}$ and $G=H_1H_2$.
 If both $H_1$ and $H_2$ are nilpotent of class at most~$k$,
 for some $k\in\mathbb{N}$, then $G$ is nilpotent of class at most~$k$.
 On the other hand, if~$G$ is nilpotent of class~$k$
 then~$H_1$ or $H_2$ are nilpotent of class at most~$k$.
\end{lemma}

\begin{proof}
We will prove by induction that, for each $i\in\mathbb{N}$,
\begin{align*}
\gamma_i(G)=\gamma_i(H_1)\gamma_i(H_2).
\end{align*}
By assumption,
\begin{align*}
\gamma_0(G)=G=H_1H_2=\gamma_0(H_1)\gamma_0(H_2),
\end{align*}
hence the statement is true for $i=0$.

Let $i>0$. Take $a\in\gamma_i(G)$ and $b\in G$. Then by the induction hypothesis, there are $a_1\in\gamma_i(H_1)$, $a_2\in\gamma_i(H_2)$, $b_1\in H_1$, $b_2\in H_2$ such that $a=a_1a_2$ and $b=b_1b_2$. Therefore,
\begin{align*}
&[a,b]=a^{-1}b^{-1}ab=a_2^{-1}a_1^{-1}b_2^{-1}b_1^{-1}a_1a_2b_1b_2=a_1^{-1}b_1^{-1}a_1b_1a_2^{-1}b_2^{-1}a_2b_2=\\
&[a_1,b_1][a_2,b_2]\in[\gamma_i(H_1),H_1][\gamma_i(H_2),H_2]=\gamma_{i+1}(H_1)\gamma_{i+1}(H_2).
\end{align*}
This means that $\gamma_{i+1}(G)\subseteq \gamma_{i+1}(H_1)\gamma_{i+1}(H_2)$.
The other inclusion uses the same argument.

The rest of the proof is now evident.
\end{proof}

\begin{theorem}\label{thm:main}
 Let $(X,\circ,\ld_{\circ},\bullet,/_{\bullet})$ be a distributive birack and let $k\geq 2$. Then
 the following conditions are equivalent:
 \begin{enumerate}
  \item[(i)] $|\mathrm{Ret}^k(X)|=1$,
  \item[(ii)] $(X,\circ,\ld_\circ,\bullet,/_{\bullet})$ is $k$-reductive,
  \item[(iii)] $(X,\circ,\ld_\circ,\bullet,/_{\bullet})$ is $k$-permutational,
  \item[(iv)] $\Mlt(X)$ is nilpotent of class at most~$k-1$.
 \end{enumerate}
\end{theorem}

\begin{proof}
Most of the claim, namely (ii)$\Leftrightarrow$(iii)$\Leftrightarrow$(iv), follows from Theorem~\ref{thm:ld_nilp_ni}, Lemmas~\ref{lem:com_lr} and \ref{lem:nilp_com}. We only have to prove (i)$\Leftrightarrow$(iv). The proof will be given by the induction on~$k$.

First, let $(X,\circ,\ld_\circ,\bullet,/_{\bullet})$ be an idempotent birack. For such biracks, (i)$\Leftrightarrow$(iv) is true even for $k=1$. It is evident that $|\mathrm{Ret}(X)|=1$ if and only if
$\Mlt(X)$ is nilpotent of class~$0$.

Let $k>1$. Similarly as in the proof of Theorem \ref{thm:ld_nilp_i}, the structure of the retract $\mathrm{Ret}(X)$ is formally defined as the birack $(\tilde X=\{(L_x,\mathbf{R}_x)\colon x\in X\},\tilde \circ,\ld_{\tilde \circ},\tilde\bullet,/_{\tilde\bullet})$ such that 
\begin{align*}
&(L_x,\mathbf{R}_x) \mathbin{\tilde \circ}(L_y,\mathbf{R}_y)=(L_{x\circ y},\mathbf{R}_{x\circ y})=(L_{x\circ y},\mathbf{R}_{y})\stackrel{\rm def}{=}\mathcal{L}_{(L_x,\mathbf{R}_x)}((L_y,\mathbf{R}_y)),\\
&(L_x,\mathbf{R}_x) \mathbin{\ld_{\tilde \circ}}(L_y,\mathbf{R}_y)=(L_{x\ld_\circ y},\mathbf{R}_{x\ld_\circ y})=(L_{x\ld_\circ y},\mathbf{R}_{y}),\\
&(L_x,\mathbf{R}_x) \mathbin{\tilde \bullet}(L_y,\mathbf{R}_y)=(L_{x\bullet y},\mathbf{R}_{x\bullet y})=(L_{x},\mathbf{R}_{x\bullet y})\stackrel{\rm def}{=}\mathcal{R}_{(L_y,\mathbf{R}_y)}((L_x,\mathbf{R}_x)),\\
&(L_x,\mathbf{R}_x) \mathbin{/_{\tilde\bullet}}(L_y,\mathbf{R}_y)=(L_{x/_{\bullet} y},\mathbf{R}_{x/_{\bullet} y})=(L_{x},\mathbf{R}_{x/_{\bullet} y}).
\end{align*}
We define the following mapping $\Phi\colon\Mlt(X)\to\Mlt(\tilde X)$: 
$$\Phi(\alpha)((L_x,\mathbf{R}_x))=(\alpha L_x\alpha^{-1},\alpha \mathbf{R}_x\alpha^{-1})=(L_{\alpha(x)},\mathbf{R}_{\alpha(y)}).$$
The mapping $\Phi$ is onto since $\Phi((L_y,\mathbf{R}_y))=\mathcal{L}_{(L_y,\mathbf{R}_y)}$. And it is a homomorphism since
$$\Phi(\alpha\beta)((L_x,\mathbf{R}_x))=(L_{\alpha\beta(x)},\mathbf{R}_{\alpha\beta(x)})=\Phi(\alpha)((L_{\beta(x)},\mathbf{R}_{\beta(x)}))=\Phi(\alpha)\Phi(\beta)((L_x,\mathbf{R}_x)).$$
Now we compute the kernel of the homomorphism:
$$\Phi(\alpha)=\mathrm{id} \Leftrightarrow \Phi(\alpha)((L_x,\mathbf{R}_x))=(L_x,\mathbf{R}_x) \Leftrightarrow  (\alpha L_x\alpha^{-1},\alpha\mathbf{R}_x\alpha^{-1})=(L_x,\mathbf{R}_x)
\Leftrightarrow \alpha\in Z(\Mlt(X)).$$
Hence $\Mlt(X)$ is nilpotent of class~$k$ if and only if $\Mlt(\mathrm{Ret}(X))$ is nilpotent
of class $k-1$. This proves the equivalence for all idempotent distributive biracks. If $(X,\circ,\ld_\circ,\bullet,/_{\bullet})$ is non-idempotent and $k\geq 2$ the proof goes the same way since by Corollary \ref{cor:sf1} the retract $\mathrm{Ret}(X)$ is idempotent and in the inductive step we do not use the idempotency.
 \end{proof}

For non-distributive biracks, there is no equivalence between multipermutation and nilpotency,
as we have remarked already in the introduction.
There is no equivalence even on a small level --
in the proof of \cite[Theorem 3.1]{CJO10}, Ced\'o, Jespers and Okni\'nski gave an example of an involutive $3$-reductive birack with an abelian permutation group. This birack is non-distributive since the only involutive and distributive biracks are $2$-reductive as shown in \cite[Corollary 5.5]{JPZ19}.

\end{document}